\documentclass[11pt,reqno]{amsart}
\usepackage{amsmath,amscd,amssymb,amsfonts,amsthm,courier,relsize,bm}
\usepackage{hyperref,enumerate,mathrsfs,mathtools,slashed}

\usepackage{xcolor}  %the next few lines set the colors of the hyperlinks, and remove the color boxes
\hypersetup{
    colorlinks,
    linkcolor={red!50!black},
    citecolor={blue!70!black},
    urlcolor={blue!80!black}
}

\newtheorem{theorem}{Theorem}[section]
\newtheorem{corollary}[theorem]{Corollary}
\newtheorem{proposition}[theorem]{Proposition}

\newtheorem{lemma}[theorem]{Lemma}
\theoremstyle{definition}    
\newtheorem{definition}[theorem]{Definition}
\theoremstyle{remark}

\newtheorem{remark}[theorem]{Remark}

\newtheorem{example}[theorem]{Example}

\newcommand{\ignore}[1]{}

\newcommand{\scr}[1]{\mathscr{#1}}
\newcommand{\mf}[1]{\mathfrak{#1}}

\newcommand{\tn}[1]{\textnormal{#1}}
\renewcommand{\i}{{\mathrm{i}}}

\def\d{\ensuremath{\mathrm{d}}}
\DeclareMathOperator*{\res}{Res}

\def\ad{\ensuremath{\textnormal{ad}}}

\def\C{\ensuremath{\mathcal{C}}}
\def\D{\ensuremath{\mathcal{D}}}
\def\E{\ensuremath{\mathcal{E}}}

\def\L{\ensuremath{\mathcal{L}}}

\def\O{\ensuremath{\mathcal{O}}}

\def\bC{\ensuremath{\mathbb{C}}}
\def\bR{\ensuremath{\mathbb{R}}}
\def\bZ{\ensuremath{\mathbb{Z}}}

\def\bA{\ensuremath{\mathbb{A}}}

\def\End{\ensuremath{\textnormal{End}}}

\def\ker{\ensuremath{\textnormal{ker}}}

\def\Tr{\ensuremath{\textnormal{Tr}}}
\def\tr{\ensuremath{\textnormal{tr}}}
\def\pr{\ensuremath{\textnormal{pr}}}
\def\dim{\ensuremath{\textnormal{dim}}}
\def\Sym{\ensuremath{\textnormal{Sym}}}
\def\dom{\ensuremath{\textnormal{dom}}}

\def\Cl{\ensuremath{\textnormal{Cl}}}

\def\index{\ensuremath{\textnormal{index}}}

\def\Ahat{\ensuremath{\widehat{\textnormal{A}}}}

\newcommand{\br}[1]{\bm{\mathrm{#1}}}

\title{Connes-Moscovici residue cocycle for some Dirac-type operators}
\author{Ahmad Reza Haj Saeedi Sadegh, Yiannis Loizides and Jesus Sanchez}
\begin{document}
\sloppy
\maketitle

\begin{abstract}
The residue cocycle associated to a suitable spectral triple is the key component of the Connes-Moscovici local index theorem in noncommutative geometry. We review the relationship between the residue cocycle and heat kernel asymptotics. We use a modified version of the Getzler calculus to compute the cocycle for a class of Dirac-type operators introduced by Bismut, obtained by deforming a Dirac operator by a closed $3$-form $B$. We also compute the cocycle in low-dimensions when the $3$-form $B$ is not closed.
\end{abstract}

\section{Introduction}
An even spectral triple consists of an algebra $A$ acting on a $\bZ_2$-graded Hilbert space $H$ and an odd (unbounded) self-adjoint operator $D$ on $H$ such that $[D,a]$ is bounded and $a(1+D^2)^{-1}$ is compact for $a \in A$. Under suitable analytic hypotheses, Connes and Moscovici \cite{connes1995local} introduced a cocycle in the $(b,B)$-bicomplex computing the periodic cyclic cohomology of $A$, involving a complicated combination of residues of spectral $\zeta$-functions for operators built from $D$, $\Delta=D^2$ and $A$. The cocycle can be used to compute index pairings with elements of the K-theory $K_0(A)$, and the resulting formula is referred to as the local index theorem in noncommutative geometry. We provide a brief introduction in Section \ref{s:ConnesMoscovici}, and explain the relationship with heat kernel asymptotics.

The first example of a spectral triple is the case $A=C^\infty(M)$ for a closed even-dimensional Riemannian spin manifold $(M^n,S \circlearrowleft \bC l(T^*M))$, and $D=D^{LC}$ is the spin Dirac operator acting on the space of $L^2$ spinors $L^2(M,S)$. In this case the Getzler calculus \cite{getzler1983pseudodifferential} can be used to show that all but the lowest terms in the formula for the residue cocycle vanish identically, and to simplify the remaining terms. In \cite[Remark II.1]{connes1995local} it is very briefly remarked that this approach leads to the formula (for the $p+1$-multilinear component):
\begin{equation} 
\label{e:ASCM}
\varphi_p(a_0,...,a_p)=\frac{(2\pi \i)^{-n/2}}{p!}\int_M a_0 \d a_1\cdots \d a_p \cdot \tn{det}^{1/2}\Big(\frac{R^{LC}/2}{\sinh(R^{LC}/2)}\Big)_{[n-p]} 
\end{equation}
where $0\le p\le n$ is even, $a_0,...,a_p\in C^\infty(M)$ and $R^{LC} \in \Omega^2(M,\mf{o}(TM))$ is the Riemannian curvature. Detailed proofs were given by Ponge \cite{ponge2003new}, Chern-Hu \cite{chern1997equivariant}, and Lescure \cite{lescure2001triplets}.

Bismut \cite{bismut1989local} studied the local index problem for operators of the form
\[ D=D^{LC}+c(B) \]
where $D^{LC}$ is the spin Dirac operator as above and $B \in \Omega^3(M)$ is a closed $3$-form (it is also possible to work with more general Clifford module bundles). Bismut showed that there is still a local index theorem in this case, where the local supertrace converges to
\[ (2\pi \i)^{-n/2}\tn{det}^{1/2}\Big(\frac{R_-/2}{\sinh(R_-/2)}\Big) \]
with $R_-$ the curvature of a certain metric connection $\nabla_-$ (depending on $B$). In \cite{bismut1989local} probabilistic methods were used to prove this result, but it was also remarked that other methods, including Getzler's method, could be adapted to this situation. In the first part of this article we spell out the details of this modification of Getzler's method, and use it to compute the residue cocycle for $D$. The result is not surprising: we obtain \eqref{e:ASCM} but with $R^{LC}$ replaced with $R_-$.

Part of the motivation for this work was to gain insight into the complicated higher terms in the Connes-Moscovici cocycle in situations where the Getzler calculus is not available and the terms do not vanish identically. Toward this goal, in the second part of the article we study the residue cocycle for $D=D^{LC}+c(B)$ in low dimensions ($n=4$, $6$) when $B$ is \emph{not} closed. In this case the Getzler calculus is not sufficient and some of the higher terms do appear. We give a complete formula for $n=4$ (Theorems \ref{t:phi0}, \ref{t:p2}) and for $\varphi_0$ in the case $n=6$ (Theorem \ref{t:p6}; in a sense $\varphi_0$ is the most involved to compute, because it depends on the greatest number of terms in the asymptotic expansion of the heat kernel). The formulas include the expected $R_-$ Chern-Weil representative of the $\Ahat$-class, plus a correction term constructed from $\d B$.

\bigskip

\noindent \textbf{Acknowledgements.} We thank Nigel Higson for suggesting this question and for many helpful conversations.

\bigskip

\noindent \textbf{Notation.} For elements $a,b$ of a $\bZ_2$-graded algebra, $[a,b]$ denotes the graded commutator. We use the convention $v^2=-|v|^2$ for the Clifford algebra $\Cl(V)$ of a real Euclidean vector bundle $(V,|\cdot|)$. If $V$ is oriented, then $\scr{B}\colon \wedge^{\tn{rk}(V)}V\rightarrow \bR$ denotes the `Berezin integral' (\cite[p.40]{BerlineGetzlerVergne}) determined by the Euclidean structure and the orientation.

\section{The Connes-Moscovici residue cocycle}\label{s:ConnesMoscovici}
In \cite{connes1995local} Connes and Moscovici introduced a powerful new $(b,B)$-cocycle into noncommutative geometry. The cocycle is associated to a spectral triple $(A,H,D)$ satisfying suitable technical hypotheses, and involves residues of spectral zeta functions for $\Delta=D^2$. It is cohomologous (in the $(b,B)$-bicomplex) to Connes' Chern character for the K-homology class $[F]\in K^0(A)$ defined by the operator $F=D(1+D^2)^{-1/2}$, and hence in particular may be used to compute index pairings (see for example \cite{ConnesBook, getzler1989chern, higson2006residue}). In this section we give a short introduction to the Connes-Moscovici cocycle following Higson's notes \cite{higson2006residue}, and then briefly discuss the relationship to heat kernel asymptotics. One simplifying assumption we shall make is to take the operator $D$ to be Fredholm. The non-Fredholm case is important in certain applications, but introduces minor technical complications, and we refer the interested reader to \cite{higson2006residue} for discussion of how to reduce to the Fredholm (even the invertible) case.

\subsection{The Connes-Moscovici theorem}
Recall (cf. \cite{ConnesBook}) that an \emph{even spectral triple} $(A,H,D)$ for an (ungraded) algebra $A$ consists of a $\bZ_2$-graded Hilbert space $H$, a representation of $A$ on $H$ by operators of even degree, and an unbounded odd self-adjoint operator $D$ such that for all $a \in A$,
\begin{enumerate}
\item $a(1+D^2)^{-1}$ is a compact operator, and
\item $a\cdot \dom(D)\subset \dom(D)$ and the commutator $[D,a]$ extends to a bounded operator on $H$.
\end{enumerate}
The standard example is $A=C^\infty_c(M)$, $M$ a complete even-dimensional Riemannian manifold, $H=L^2(M,S)$ and $D$ a Dirac operator acting on sections of a $\bZ_2$-graded Clifford module $S$.

The Connes-Moscovici result applies to spectral triples satisfying some technical hypotheses that we now describe. Following \cite{higson2006residue}, we introduce an algebra of `generalized differential operators' for $(A,H,D)$. Let $\Delta=D^2$ and define
\[ H^\infty=\bigcap_{s=1}^\infty \dom(\Delta^s).\]
Suppose $a\cdot H^\infty \subset H^\infty$ for all $a \in A$. Let $\D(A,D)$ be the smallest algebra of operators on $H^\infty$ that contains $A$, $[D,A]$ and such that if $X \in \D(A,D)$ then $[\Delta,X]\in \D(A,D)$. 

We suppose the algebra $\D(A,D)$ is equipped with an increasing filtration $\D(A,D)=\cup_{q \ge 0} \D_q(A,D)$, where the operators $X$ in $\D_q(A,D)$ are said to have `(generalized) order $q$' (notation: $o(X)=q$), having the following properties:
\begin{enumerate}
\item There is an even integer $r\ge 2$ such that 
\begin{equation} 
\label{e:commutatororders}
[\Delta,\D_q(A,D)]\subset \D_{q+r-1}(A,D), \qquad [D,a] \in \D_{\frac{r}{2}-1}(A,D), \quad \forall a \in A.
\end{equation}
\item If $X \in \D_q(A,D)$ then there is a constant $\varepsilon>0$ such that
\begin{equation} 
\label{e:ellipticestimate}
\varepsilon\|Xv\|\le \|v\|+\|\Delta^{\frac{q}{r}}v\|.
\end{equation}
\end{enumerate}
The integer $r$ should be thought of as the generalized order of $\Delta$, even though $\Delta$ need not itself lie in the algebra $\D(A,D)$. The correct choice of $r$ depends on the details of the application; in the standard example mentioned above $r=2$, but for example Connes and Moscovici study an interesting example with $r=4$.

We now make the simplifying assumption that $D$ is a Fredholm operator (cf. \cite[Section 6.1]{higson2006residue} for a discussion of how to reduce to this situation). For $\tn{Re}(z)>0$ define
\[ \Delta^{-z}=\int_{\C} \lambda^{-z}(\lambda-\Delta)^{-1}\d \lambda,\]
where the contour $\C$ is a downward-oriented vertical line in the complex plane which separates $0$ from the strictly positive part of the spectrum of $\Delta$.

\begin{definition}
The algebra $\D(A,D)$ has \emph{finite analytic dimension} if there is a real number $n\ge 0$ such that if $X \in \D_q(A,D)$ then for all $z \in \bC$ with real part $\Re(z)>\frac{q+n}{r}$, the operator $X\Delta^{-z}$ extends by continuity to a trace-class operator on $H$. The minimal $n\ge 0$ with this property is called the \emph{analytic dimension} of $\D(A,D)$. The function $z\mapsto \Tr(X\Delta^{-z})$ is then well-defined and holomorphic in the right half-plane $\Re(z)>\frac{q+n}{r}$. We say that $\D(A,D)$ has the \emph{analytic continuation property} if this function extends to a meromorphic function on $\bC$.
\end{definition}
If $\Delta=D^2$ for a Dirac operator $D$ on a compact Riemannian manifold $M$, the above analytic continuation property is well-known (originally due to Minakshisundaram-Pleijel) and the analytic dimension coincides with the dimension of $M$. For further context see for example the account in \cite{higson2004meromorphic}, which describes a fairly general analytic continuation result. Note in particular that if $\D(A,D)$ has finite analytic dimension $n$ then for $q>\frac{r}{n}$ and for all $a \in A \subset \D_0(A,D)$, $a(1+D^2)^{-q}$ is trace class, hence $(A,H,D)$ is a \emph{finitely summable} spectral triple and Connes' Chern character is defined (cf. \cite{ConnesBook}).

The following is a case of the Connes-Moscovici local index formula \cite{connes1995local}, although we have formulated the hypotheses in the language of \cite{higson2006residue}.
\begin{theorem}[Connes-Moscovici \cite{connes1995local}]
\label{t:CM}
Let $(A,H,D)$ be an even spectral triple with $D$ Fredholm, and let $\Delta=D^2$. Suppose $A$ preserves $H^\infty=\cap_{s\ge 1} \dom(\Delta^s)$, and define the filtered algebra $\D(A,D)$ as above. Assume $\D(A,D)$ satisfies \eqref{e:commutatororders}, \eqref{e:ellipticestimate}, has finite analytic dimension $n$, and satisfies the analytic continuation property. Then there is an even $(b,B)$-cocycle $\Phi=(\varphi_0,\varphi_2,...,\varphi_{2\lfloor \frac{n}{2}\rfloor})$ cohomologous to Connes' Chern character, which is given, for all $a_0,...,a_p \in A$, by the following expressions:
\[ \varphi_0(a_0)=\res_{z=0}(\Gamma(z)\Tr_s(a_0\Delta^{-z}))+\Tr_s(a_0\Pi)\]
where $\Pi$ is orthogonal projection onto the kernel of $D$ and $\Tr_s$ denotes the supertrace, and if $p>0$,
\begin{align*} 
\varphi_p(a_0,...,a_p)&=\sum_{|k|\le n-p} c_{pk}\res_{z=\frac{p}{2}+|k|}\Tr_s(P_k(a_0,...,a_p)\Delta^{-z}),\\ 
P_k(a_0,...,a_p)&=a_0[D,a_1]^{(k_1)}\cdots [D,a_p]^{(k_p)}
\end{align*}
where $k=(k_1,...,k_p) \in (\bZ_{\ge 0})^p$ is a multi-index, $X^{(\ell)}=\ad_{\Delta}^\ell(X)$, and 
\[ c_{pk}=\frac{(-1)^k}{k!}\frac{\Gamma(\frac{p}{2}+|k|)}{(k_1+1)(k_1+k_2+2)\cdots (k_1+\cdots+k_p+p)}, \qquad |k|=k_1+\cdots+k_p.\]
\end{theorem}
We will refer to $\Phi$ as the \emph{Connes-Moscovici cocycle} or the \emph{residue cocycle}. 
\begin{remark}
\label{r:genorders}
The bounds $p\le 2\lfloor \frac{n}{2} \rfloor$ and $|k|\le n-p$ come from generalized order considerations: the operator $P_k(a_0,...,a_p) \in \D(A,D)$ has generalized order at most $|k|(r-1)+p(\frac{r}{2}-1)$, hence by the analytic continuation assumption, $\frac{p}{2}+|k|$ will lie in the half plane where the supertrace is holomorphic provided that $n<p+|k|$.
\end{remark}

Being cohomologous to Connes' Chern character, the residue cocycle yields a formula for the index of the operator in the spectral triple, and more generally for the index pairing with an element in the K-theory group $K_0(A)$. A multi-linear functional $\varphi_p\colon A^{\otimes (p+1)}\rightarrow \bC$ may be extended to $M_k(A)^{\otimes (p+1)}$ by defining
\[ \varphi_p(m_0\otimes a_0,..., m_p\otimes a_p)=\tr(m_0\cdots m_p)\varphi_p(a_0,...,a_p).\]
For the following see for example \cite[Proposition 1.1, Theorem D]{getzler1989chern} and \cite[Theorem 2.27]{higson2006residue}.
\begin{theorem}
\label{t:CMindex}
Let $(A,H,D)$, $\Phi$ be as in Theorem \ref{t:CM} and assume in addition that $A$ is unital. Let $e \in M_k(A)$ be an idempotent. Then
\[ \tn{index}(e(D\otimes 1_k)e)=\varphi_0(e)+\sum_{p >0, \text{ even}} (-1)^{p/2}\frac{p!}{(p/2)!}\varphi_p(e-\tfrac{1}{2},e,...,e).\]
\end{theorem}

\subsection{Residues and heat kernel asymptotics}
Throughout this subsection we assume $(A,H,D)$ are as in Theorem \ref{t:CM}. The residue cocycle is closely related to the small time behavior of the heat kernel $e^{-t\Delta}$. Let $\Pi^\perp=1-\Pi$ be projection onto the orthogonal complement of $\ker(D)=\ker(\Delta)$. The following is well-known, but we include it for completeness.

\begin{proposition}
\label{p:tasymptotics}
For all $t>0$ the operator $P_k(a_0,...,a_p)e^{-t\Delta}$ is trace class. As $t \rightarrow 0^+$ the trace norm is $O(t^{-s})$ for any $s>\frac{n}{r}+|k|(1-\frac{1}{r})+p(\frac{1}{2}-\frac{1}{r})$. As $t \rightarrow \infty$ the trace norm of $P_k(a_0,...,a_p)e^{-t\Delta}\Pi^\perp$ decays exponentially.
\end{proposition}
\begin{proof}
Let $s \in \bR$ be as in the statement. By order considerations (see Remark \ref{r:genorders}), the operator $P_k(a_0,...,a_p)(1+\Delta)^{-s}$ is trace class. On the other hand by functional calculus $(1+\Delta)^s e^{-t\Delta}$ is a bounded operator with norm at most $C_s t^{-s} e^t$, where $C_s$ is a constant. It follows that $P_k(a_0,...,a_p)e^{-t\Delta}=P_k(a_0,...,a_p)(1+\Delta)^{-s}(1+\Delta)^s e^{-t\Delta}$ is trace class for all $t>0$, and that its trace norm has the claimed asymptotic behavior as $t \rightarrow 0^+$. For $t>1$,
\[ P_k(a_0,...,a_p)e^{-t\Delta}\Pi^\perp=P_k(a_0,...,a_p)e^{-\Delta}e^{-(t-1)\Delta}\Pi^\perp \]
thus the trace norm is bounded by the trace norm of $P_k(a_0,...,a_p)e^{-\Delta}$ times the operator norm $\|\Pi^\perp e^{-(t-1)\Delta}\Pi^\perp\|=e^{-(t-1)b}$ where $b>0$ is the lower bound of $\Pi^\perp\Delta\Pi^\perp$ on $\Pi^\perp H$.
\end{proof}

\begin{proposition}
\label{p:heatkernelversion}
Suppose the supertrace admits an asymptotic expansion as $t \rightarrow 0^+$ of the form
\[ \Tr_s(P_k(a_0,...,a_p)e^{-t\Delta})\sim t^{-N_k}\sum_{s\ge 0} t^{s} \psi_{k,s}(a_0,...,a_p). \]
Then the components of the residue cocycle are given by
\[ \varphi_p(a_0,...,a_p)=\sum_{|k|\le n-p} c_{pk}' \psi_{k,s_{p,k}}(a_0,...,a_p), \qquad s_{p,k}=N_k-\frac{p}{2}-|k|,\]
where
\[ c_{pk}'=\frac{(-1)^k}{k!\cdot (k_1+1)(k_1+k_2+2)\cdots (k_1+\cdots+k_p+p)}.\]
\end{proposition}
\begin{proof}
For brevity let $P_k=P_k(a_0,...,a_p)$. Since $\Delta^{-z}$ was defined to vanish on $\ker(\Delta)$, one has $\Delta^{-z}=\Delta^{-z}\Pi^\perp$. By the Mellin transform,
\[ \Gamma(z)P_k\Delta^{-z}=\int_0^\infty t^{z-1}P_ke^{-t\Delta}\Pi^\perp \d t \]
where the integral converges in the trace norm when $\Re(z)\gg 0$ by Proposition \ref{p:tasymptotics}. Therefore taking the trace and splitting the range of integration, we have for $\Re(z)\gg 0$,
\begin{equation} 
\label{e:int1}
\Gamma(z)\Tr_s(P_k\Delta^{-z})=\int_0^1 t^{z-1}\Tr_s(P_ke^{-t\Delta}\Pi^\perp) \d t+\int_1^\infty t^{z-1}\Tr_s(P_ke^{-t\Delta}\Pi^\perp) \d t
\end{equation}
Proposition \ref{p:tasymptotics} implies further that the integral over $(1,\infty)$ is a holomorphic function $f_k(z)$. Substituting $\Pi^\perp=1-\Pi$, $e^{-t\Delta}\Pi=\Pi$ in \eqref{e:int1} yields 
\begin{equation} 
\label{e:int2}
\Gamma(z)\Tr_s(P_k\Delta^{-z})=\int_0^1 t^{z-1}\Tr_s(P_ke^{-t\Delta}) \d t-\frac{1}{z}\Tr_s(P_k\Pi)+f_{k}(z).
\end{equation} 
Using the asymptotic expansion of $\Tr_s(P_ke^{-t\Delta})$ as $t \rightarrow 0^+$, the integral on the right hand side has an analytic continuation (with simple poles) to a neighborhood of $\frac{p}{2}+|k|$, namely
\[ \sum_{s=0}^{M} \frac{1}{s-N_k+z}\psi_{k,s}+\int_0^1 t^{z-1} R_M(t)\d t \]
for any $M\ge \tn{max}(0,N_k-\frac{p}{2}-|k|+1)$ where $R_M(t)=o(t^{M-N_k})$ is the remainder. The analytic continuations of the two sides of \eqref{e:int2} must agree near $z=\frac{p}{2}+|k|$ and consequently we may take residues of both sides, which gives the result. Note in particular that when $p>0$ the second term on the right hand side of \eqref{e:int2} is holomorphic near $z=\frac{p}{2}+|k|$ so does not contribute to the residue. When $p=0$ (so $k=\emptyset$), $P_{\emptyset}=a_0$ and the residue of the second term is $-\Tr_s(a_0\Pi)$.
\end{proof}

\section{Dirac-type operators and Getzler order}
This section is mostly expository. We describe Bismut's generalization \cite{bismut1989local} of the Lichnerowicz formula to Dirac-type operators $D=D^{LC}+c(B)$. We then specialize to the case where $B$ is a $3$-form and briefly introduce a Getzler symbol calculus `adapted to $B$', which is a slight variation of the usual Getzler calculus. The observation that a variation of the usual Getzler calculus is appropriate  for studying the heat operator $e^{-tD^2}$ is due to Bismut \cite{bismut1989local}. Our discussion of Getzler calculus draws from the approaches in \cite{BerlineGetzlerVergne, roe1998elliptic}, and as in these references, we will only need a less elaborate version of Getzler's original calculus \cite{getzler1983pseudodifferential}, sufficient for handling compositions of the form $P\circ Q$ where $P$ is a differential operator and $Q$ is a smoothing operator. Throughout we will work on a closed Riemannian spin manifold, the extension to operators acting on general Clifford modules being well understood (cf. \cite{BerlineGetzlerVergne}).

\subsection{Dirac-type operators}
Let $(M^n,g)$ be a closed Riemannian spin manifold with $\bZ_2$-graded spinor bundle $S$, and let $c\colon \bC l(T^*M) \xrightarrow{\sim} \End(S)$ denote the Clifford action. Equip $S$ with a Hermitian structure such that for $v \in T^*M$, $c(v)$ is skew-Hermitian. There is a canonical isomorphism (the \emph{Clifford symbol map}) of $\bZ_2$-graded complex vector spaces $\bC l(T^*_xM) \simeq \wedge T^*_xM_\bC$ that sends $e_{i_1}\cdots e_{i_k} \in \Cl(T^*_xM)$ to $e_{i_1}\wedge \cdots \wedge e_{i_k} \in \wedge T^*_xM$, where $e_1,...,e_n$ is any orthonormal frame of $T^*_xM$ (cf. \cite{BerlineGetzlerVergne}), and we will use this isomorphism to identify $\bC l(T^*M)$ and $\wedge T^*M_\bC$.

The Levi-Civita connection $\nabla^{LC}$ on $TM\simeq T^*M$ determines a canonical connection (the spin connection) on $S$ that we also denote by $\nabla^{LC}$. \emph{The spin Dirac operator} $D^{LC}$ acting on smooth sections of $S$ is the odd, essentially self-adjoint, first-order differential operator given by the composition
\begin{equation} 
\label{e:Dirac}
\Gamma(S)\xrightarrow{\nabla^{LC}} \Gamma(T^*M\otimes S) \xrightarrow{c}\Gamma(S).
\end{equation}
More generally a \emph{Dirac operator} acting on smooth sections of $S$ is an operator given by a composition similar to \eqref{e:Dirac}, but allowing the spin connection $\nabla^{LC}$ to be replaced with a connection of the form $\nabla^{LC}+\sqrt{-1}a$ where $a$ is a $\bR$-valued $1$-form (equivalently, twist $S$ by a trivial Hermitian line bundle and couple $D^{LC}$ to it using a possibly non-trivial Hermitian connection). In this article by a \emph{Dirac-type operator} we shall mean an odd, essentially self-adjoint first order differential operator $D$ acting on smooth sections of $S$ that differs from $D^{LC}$ by a smooth bundle endomorphism. The most general such operator is of the form
\[ D=D^{LC}+c(B) \]
where $B \in \Gamma(\wedge T^*M_\bC)$ is a differential form of odd (possibly mixed) degree satisfying $c(B)^*=c(B)$. When $B$ is a $1$-form, $D$ is a Dirac operator in the above sense.

For any $B$ as above, there is a spectral triple $(C^\infty_c(M),L^2(M,S),D)$ satisfying all the hypotheses of the Connes-Moscovici theorem. All of these spectral triples represent the same element in the K-homology group of the closed manifold $M$, and hence although their residue cocycles will differ, they are guaranteed to be $(b,B)$-cohomologous. The operator $\Delta=D^2$ is a generalized Laplacian in the sense of \cite[Chapter 2]{BerlineGetzlerVergne}, hence the heat kernel $e^{-t\Delta}$ has an asymptotic expansion as $t \rightarrow 0^+$ to which Proposition \ref{p:heatkernelversion} applies. 

\subsection{Bismut's Lichnerowicz formula}
Let $B \in \Gamma(\wedge T^*M_\bC)$ be a differential form of odd degree such that $c(B)^*=c(B)$. Define a new connection $\nabla \colon \Gamma(S)\rightarrow \Gamma(T^*M\otimes S)$ on $S$ given by
\begin{equation} 
\label{e:adaptedconnection}
\nabla_X=\nabla^{LC}_X+c(\iota_X B), \qquad X \in \mf{X}(M) 
\end{equation}
where $\iota_X$ denotes contraction with the vector field $X$. (To avoid a possible misconception, we mention that if $e_1,...,e_n$ is a local orthonormal frame, then $\sum_i c(e_i)\nabla_{e_i}=D^{LC}+k\,c(B)$ if $B \in \Gamma(\wedge^k T^*M_\bC)$, which is \emph{not} the operator $D$ unless $k=1$. The operator $D$ is however the Dirac operator associated to the Clifford superconnection $\bA=\nabla^{LC}+B$ in the sense of \cite[p.116]{BerlineGetzlerVergne}.) The formal adjoint of $\nabla$ with respect to the Riemannian $L^2$ inner products is denoted $\nabla^* \colon \Gamma(T^*M\otimes S)\rightarrow \Gamma(S)$.

The odd differential form $B$ gives rise to a collection $B_j$, $j=1,2,3,...$ of even degree differential forms (with degrees $(k-2j-1)^2$ if $B$ has odd degree $k$), given in terms of a local orthonormal frame $e_1,...,e_n$ by
\[ B_j=\sum_{i_1<\cdots <i_{2j+1}} \big(\iota(e_{i_1})\cdots \iota(e_{i_{2j+1}})B\big)^2. \]
(The result is independent of the choice of local orthonormal frame.) Bismut proved the following formula for the square of the Dirac-type operator $D=D^{LC}+c(B)$.
\begin{proposition}[\cite{bismut1989local}, Theorem 1.1]
\label{p:BismutLichnerowicz}
Let $D=D^{LC}+c(B)$ where $D^{LC}$ is the spin Dirac operator and $B$ is an odd differential form. Let $\Delta=D^2$. Then
\[ \Delta=\nabla^*\nabla+\frac{\kappa}{4}+c(\d B)+2\sum_{j\ge 1} (-1)^{j}jc(B_j),\]
where $\kappa$ is the scalar curvature. When $B \in \Gamma(\wedge^3 T^*M)$ the formula simplifies to
\[ \Delta=\nabla^*\nabla+\frac{\kappa}{4}+c(\d B)-2|B|^2.\]
\end{proposition}
In terms of a local orthonormal frame $e=\{e_1,...,e_n\}$,
\[ \nabla^*\nabla=-\Big(\sum_{i=1}^n \nabla_{e_i}^2\Big)-\nabla_{\nu_e}, \qquad \nu_e=\sum_{i=1}^n \nabla^{LC}_{e_i}e_i=-\sum_{i=1}^n \tn{div}_g(e_i)e_i.\]
The formula in Proposition \ref{p:BismutLichnerowicz} appears slightly simpler than \cite[Theorem 1.1]{bismut1989local} because we have omitted twists by an auxiliary bundle.

Throughout the rest of the article, we restrict to the case where $B$ is a $3$-form. The condition $c(B)^*=c(B)$ implies $B \in \Gamma(\wedge^3 T^*M)$ is a real $3$-form. A special feature of the $3$-form case is that the connection $\nabla$ is the lift (via the isomorphism $\mf{o}_n=\mf{so}_n\simeq \mf{spin}_n$) of a metric connection on the tangent bundle---also denoted $\nabla$ when there is no risk of confusion---given by the formula
\begin{equation} 
\label{e:nablaM}
\nabla=\nabla^{LC}+B_{\mf{o}}, \qquad B_{\mf{o}} \in \Omega^1(M,\mf{o}(TM)), %\quad g(B_{\mf{o}}(X)Y,Z)=2B(X,Y,Z).
\end{equation}
where for any $A \in \Omega^k(M)$, $k\ge 2$ we define $A_{\mf{o}} \in \Omega^{k-2}(M,\mf{o}(TM))$ using the metric:
\[ g(A_{\mf{o}}(X_1,...,X_{k-2})X_{k-1},X_k)=2A(X_1,...,X_k).\]
The lift property ensures that
\[ [\nabla_X,c(Y)]=c(\nabla_X Y). \]
The extra skew-symmetry of \eqref{e:nablaM} implies that the torsion $T_\nabla$ of $\nabla$ is skew-symmetric,
\[ g(T_\nabla(X,Y),Z)=4B(X,Y,Z),\]
and $\nabla_X X=\nabla^{LC}_X X$ for all $X \in \mf{X}(M)$, hence $\nabla, \nabla^{LC}$ have the same geodesics.
\begin{remark}
When $\d B=0$ there is an interesting perspective on $\nabla$ coming from generalized geometry in the sense of Hitchin, where $\nabla$ may be thought of as the analogue of the Levi-Civita connection when doing geometry with a closed $3$-form `background', see for example \cite{hitchin2010lectures, gualtierithesis}.
\end{remark}

\subsection{Getzler calculus}\label{s:Getzler}
Let $B \in \Gamma(\wedge^3 T^*M)$ be a $3$-form on $M$, and let $\nabla$ denote the corresponding connection on $TM$ (equation \eqref{e:nablaM}) or its lift to $S$ (equation \eqref{e:adaptedconnection}). A key feature of the Getzler symbol calculus applied to the spin Dirac operator $D^{LC}$ is that, by the Lichnerowicz formula, the square $(D^{LC})^2$ has order $2$. The appearance of the connection $\nabla$ in Proposition \ref{p:BismutLichnerowicz} suggests that in the study of heat kernel asymptotics for the operator $D$, a variation of the Getzler calculus that replaces the Levi-Civita connection $\nabla^{LC}$ with $\nabla$ should be used.

Let
\[ \pr \colon TM \rightarrow M, \qquad s \colon M \times M \rightarrow M,  \]
be the bundle projection, resp. projection map to the second factor (i.e. the source map of the pair groupoid $M \times M$). The Riemannian exponential map identifies a tubular neighborhood $\bm{U}$ of the $0$-section in $TM$ with a neighborhood $U$ of the diagonal in $M \times M$:
\begin{equation} 
\label{e:RiemExp}
\exp\colon v\in \bm{U}\subset TM \mapsto (\exp_y(v),y) \in U \subset M \times M 
\end{equation}
where $v \in T_yM$. The inverse of the diffeomorphism \eqref{e:RiemExp} is denoted 
\begin{equation}
\label{e:RiemExpInv}
\br{x}\colon U \rightarrow \bm{U}\subset TM. 
\end{equation}
The maps \eqref{e:RiemExp}, \eqref{e:RiemExpInv} intertwine $s|_U$, $\pr|_{\br{U}}$. For $y \in M$ let $U_y=U\cap (M\times \{y\})$, a geodesic ball around $y$, and let $\br{x}_y \colon U_y \rightarrow T_yM$ be the restriction of $\br{x}$ to $U_y$.

Using $\nabla$-parallel translation along radial geodesics, followed by the Clifford symbol map, we obtain isomorphisms
\begin{equation}
\label{e:synch}
S\boxtimes S^*|_U \simeq s^*(S\otimes S^*)=s^*\End(S)=s^*\bC l(T^*M)\simeq s^*\!\wedge\!T^*M_\bC.
\end{equation}
Using the inverse of the exponential map \eqref{e:RiemExpInv} on the base combined with the isomorphism \eqref{e:synch} on the fibres, we obtain an isomorphism
\begin{equation}
\label{e:isosections}
\Gamma(S\boxtimes S^*|_U) \simeq \Gamma(\pr^*\!\!\wedge\!T^*M_\bC|_{\bm{U}}),
\end{equation}
that will be used frequently below.

\begin{definition}
\label{d:$r$-fibreddo}
An \emph{$s$-fibred differential operator} on $\Gamma(S\boxtimes S^*|_U)$ is differential operator $T$ with smooth coefficients on $\Gamma(S\boxtimes S^*|_U)$ given by a family $\{T_y\}_{y \in M}$, where $T_y$ is a differential operator acting on $\Gamma(S|_{U_y}\boxtimes S_y^*)$. The space of $s$-fibred differential operators forms an algebra under composition. An $s$-fibred differential operator $T$ is said to \emph{vanish on the diagonal} if $T(\Gamma(S\boxtimes S^*|_U))$ is contained in the subspace of $\Gamma(S\boxtimes S^*|_U)$ consisting of sections that vanish on the diagonal. Using the identification \eqref{e:isosections}, $T$ yields a differential operator
\[ \bm{T}\colon \Gamma(\pr^*\!\!\wedge\!T^*M_\bC|_{\bm{U}}) \rightarrow \Gamma(\pr^*\!\!\wedge\!T^*M_\bC|_{\bm{U}}) \]
given by a family $\{\bm{T}_y\}_{y \in M}$ of differential operators along the fibres of $\bm{U} \rightarrow M$.
\end{definition}
We mention several examples that will appear frequently below.
\begin{example}
\label{ex:diffopfibred}
Any differential operator $P$ \emph{on $M$} acting on sections of $S$ determines an $s$-fibred differential operator by `copying' $P$ on each $M \times \{y\}\subset M \times M$, and using the fact that $s^*S^*|_{M\times \{y\}}\simeq M\times S_y^*$ is canonically trivial to extend $P$ to act on sections of $S\boxtimes S^*|_{M\times \{y\}}=S\boxtimes S_y^*$.
\end{example}
\begin{example}
\label{ex:varrho}
Smooth functions on $U$ act by multiplication on $\Gamma(S\boxtimes S^*|_U)$ and hence determine $s$-fibred differential operators. One example that appears frequently below is the function $\br{x}^2=g(\br{x},\br{x})$ giving the squared distance to the diagonal. Further examples are the powers $\varrho^r$ for $r \in \bR$, where by definition the restriction $\varrho^r_y$ of $\varrho^r$ to the normal coordinate chart $U_y=U\cap (M\times \{y\})$ is
\[ \varrho_y^r(\br{x}_y)=\det(g_{ab})^{\frac{r}{4}} \]
where $g_{ab}=g(\partial_a,\partial_b)$ are the components of the metric in the normal coordinate system on $U_y$.
\end{example}
\begin{example}
Generalizing the previous example, any smooth section $Q \in \Gamma(S\boxtimes S^*|_U)$ determines an $s$-fibred differential operator, via the identification $\Gamma(S\boxtimes S^*|_U)\simeq \Gamma(s^*\End(S))$, and letting $Q|_U$ act on sections of $\Gamma(S\boxtimes S^*|_U)\simeq \Gamma(\pr_2^*\End(S))$ by pointwise composition of endomorphisms of $S$.
\end{example}
\begin{example}
\label{ex:Euler}
Let $\E$ be the vector field on $U \simeq \bm{U}$ corresponding to the Euler vector field on $\bm{U}\subset TM$ under the exponential map, whose integral curves are radial geodesics. Then $\nabla_\E$ is an $s$-fibred differential operator that vanishes on the diagonal. The identification \eqref{e:synch} trivializes the bundle $S$ along radial geodesics, and hence instead of $\nabla_\E$ we will often simply write $\E$. %Working in normal coordinates, the operator on $U_y=U\cap (M\times \{y\})$ in the corresponding family is
%\[ \E_y=\sum_a \br{x}_a\frac{\partial}{\partial \br{x}_a}=\sum_a \br{x}_a\partial_a.\]
\end{example}

The inclusion $TM\times \bC=\pr^*\!\!\wedge^0T^*M_\bC \hookrightarrow \pr^*\!\!\wedge\!T^*M_\bC$ induces a map on sections $\eta \colon C^\infty(TM)\hookrightarrow \Gamma(\pr^*\!\!\wedge\!T^*M_\bC)$. If $\bm{T}\colon \Gamma(\pr^*\!\!\wedge\!T^*M_\bC|_{\bm{V}})\rightarrow \Gamma(\pr^*\!\!\wedge\!T^*M_\bC|_{\bm{V}})$ is a differential operator defined on some open set $\bm{V} \subset TM$, then the composition
\[ \bm{T}\circ \eta \colon C^\infty(\bm{V}) \rightarrow \Gamma(\pr^*\!\!\wedge\! T^*M_\bC|_{\bm{V}}) \]
is again a differential operator on $\bm{V}$.

\begin{definition}[compare \cite{roe1998elliptic} exercises 12.31, 12.32 and \cite{BerlineGetzlerVergne} pp. 156--157]
For a section $\alpha \in \Gamma(S\boxtimes S^*|_U)\simeq \Gamma(\pr^*\!\!\wedge\!T^*M_\bC|_{\bm{U}})$ and $u \in \bR_{>0}$ define $\delta_u \alpha \in \Gamma(\pr^*\!\!\wedge\!T^*M_\bC|_{u^{-1}\bm{U}})$ by
\begin{equation} 
\label{e:Gres}
\delta_u \alpha(\br{x})=\sum_{k=0}^n u^{-k}\alpha_{[k]}(u\br{x}),
\end{equation}
where $\alpha_{[k]}$ is the component lying in the $k$-th exterior power. The operation $\delta_u$ is the \emph{Getzler re-scaling for the connection} $\nabla$. For an $s$-fibred differential operator $T$ and $m \in \bZ$ define
\begin{equation} 
\label{e:Grescaling}
\sigma^G_m(T)=\lim_{u\rightarrow 0^+} u^m \delta_u\circ \bm{T}\circ\delta_u^{-1}\circ \eta \in \mf{D}(TM)\otimes \wedge T^*M
\end{equation}
when the limit exists, where $\mf{D}(TM) \rightarrow M$ is the bundle of algebras whose fibre over $y \in M$ consists of differential operators on $T_yM$ with polynomial coefficients. In \eqref{e:Grescaling}, $\delta_u \circ \bm{T}\circ \delta_u^{-1}$ is to be viewed as a differential operator on $u^{-1}\cdot \bm{U}\subset TM$, which in the limit produces a differential operator on $TM$. If the limit exists for some $m \in \bZ$ then $T$ is said to have \emph{Getzler order} $m$ (notation: $o^G(T)=m$), and \eqref{e:Grescaling} is the $m$-th order \emph{Getzler symbol} of $T$.  The \emph{constant term} of the Getzler symbol, denoted $\sigma^{G,0}_{m}(T)$ is the element of $\Sym(TM)\otimes \wedge T^*M$ (where $\Sym(TM)$ should be thought of as the bundle of constant coefficient differential operators along the fibres of $TM$) obtained by evaluating the polynomial coefficients of $\sigma^G_m(T)$ along the zero section. Getzler order determines a filtration of the algebra of $s$-fibred differential operators, and the Getzler symbol satisfies (cf. \cite[Proposition 12.22]{roe1998elliptic})
\[ \sigma^G_{m_1+m_2}(T_1T_2)=\sigma^G_{m_1}(T_1)\sigma^G_{m_2}(T_2) \]
with $T_1,T_2$ being $s$-fibred differential operators with Getzler orders $o^G(T_i)=m_i$.
\end{definition}
\begin{remark}
The small deviation from the usual setup was that $\nabla$-parallel translation was used in \eqref{e:synch}, \emph{not} $\nabla^{LC}$-parallel translation.
\end{remark}
To clarify the definition and for use in later calculations, we describe a number of examples.
\begin{example}
\label{ex:Getzlerfunctions}
Let $f \in C^\infty(U)$ be a smooth function that vanishes to order $k$ on the diagonal.  Then $o^G(f)=-k$ and
\[ \sigma^G_{-k}(f)=\lim_{u\rightarrow 0} u^{-k}f(u\br{x}) \]
is a smooth function on $TM$ which is polynomial along the fibres of $TM \rightarrow M$ (it is the homogeneous $k$-th order Taylor polynomial of $f$ in directions normal to the diagonal). In particular a function always possesses a $0$-th order Getzler symbol $\sigma^G_0(f)=\pr^*(f|_{\tn{Diag}_M})$.
\end{example}
\begin{example}
\label{ex:forms}
Let $\alpha=\alpha_1\cdots \alpha_k \in \Gamma(\wedge^k T^*M)$ be a decomposable $k$-form, and let $c(\alpha)$ be the $0$-th order operator on $S$ given by the Clifford action. We view $c(\alpha)$ as a $s$-fibred differential operator as in Example \ref{ex:diffopfibred}. Under the isomorphism $S\otimes S^*\simeq \wedge T^*M_\bC$, Clifford multiplication by $\alpha_j \in \wedge^1T^*M_\bC$ becomes $\epsilon(\alpha_j)-\iota(\alpha_j)$, where $\epsilon$ (resp. $\iota$) denotes exterior multiplication (resp. contraction). It follows that $o^G(c(\alpha))=k$ and
\[ \sigma^G_k(c(\alpha))=1 \otimes \alpha \in \mf{D}(TM)\otimes \wedge T^*M.\]
\end{example}
\begin{example}
Generalizing the previous two examples, any smooth section $Q \in \Gamma(S\boxtimes S^*|_U)$ has Getzler order (at most) $n=\dim(M)$. Its $n$-th order Getzler symbol is the $n$-form part of its restriction to the diagonal.
\end{example}
\begin{example}
The Getzler order $o^G(\E)=0$ and
\[ \sigma^G_0(\E)=\E\otimes 1\]
is the Euler vector field on $TM$.
\end{example}
\subsubsection{Getzler symbol of covariant derivatives}
Let $y \in M$, let $e_1,...,e_n$ be an orthonormal basis of $T_yM$, and let $\br{x}_{y,a}=g(\br{x}_y,e_a)$ be the corresponding normal coordinates on $U_y$. By $\nabla$-parallel translation along radial geodesics, extend $e_1,...,e_n$ to an orthonormal frame of $TM|_{U_y}$. Let $\omega=\sum_a \omega_a \d \br{x}_{y,a} \in \Omega^1(U_y,\mf{o}_n)$ be the connection $1$-form for $\nabla$ on $U_y$ relative to the frame $e_1,...,e_n$. By construction
\begin{equation} 
\label{e:eulercontraction}
\iota(\E)\omega=0,
\end{equation}
and using the Cartan formula, the Lie derivative
\begin{equation} 
\label{e:lieomega}
\L_{\E} \omega=\iota(\E)R
\end{equation}
where $R \in \Omega^2(U_y,\mf{o}_n)$ is the curvature. Since $\omega$ vanishes at $y$ (where $\br{x}_y=0$), equation \eqref{e:lieomega} implies (cf. \cite[Proposition 1.18]{BerlineGetzlerVergne}),
\begin{equation} 
\label{e:omegalinear}
\omega_a(\br{x}_y)=-\frac{1}{2}\sum_b R(\partial_a,\partial_b)_y\br{x}_{y,b}+O(\br{x}_y^2), 
\end{equation}
where $O(\br{x}_y^2)$ denotes an $\mf{o}_n$-valued smooth function vanishing to order $2$ at the point $y$. Via the isomorphism $\mf{o}_n=\mf{so}_n\simeq \mf{spin}_n$, the operator $\nabla_{\partial_a}$ on $S$ is
\[ \nabla_{\partial_a}=\partial_a+\frac{1}{8}\sum_{b,c,d} (R_{cdab})_y\br{x}_{y,b}c(e_c)c(e_d)+O(\br{x}_y^2). \]
Therefore if $X \in \mf{X}(M)$ is a vector field, $X_y=\sum_a X_a \partial_a$, then near $y \in M$,
\[ \nabla_X=X+\frac{1}{8}\sum_{a,b,c,d} (R_{cdab})_yX_a\br{x}_{y,b}c(e_c)c(e_d)+O(\br{x}_y^2). \]
It follows that (compare \cite[Proposition 4.20]{BerlineGetzlerVergne}), at the point $y \in M$,
\begin{equation} 
\label{e:sigmanablay}
\sigma_1^G(\nabla_X)_y=X_y\otimes 1+\frac{1}{8}\sum_{a,b,c,d} (R_{cdab})_yX_a\br{x}_{y,b}\otimes e_ce_d.
\end{equation}
The sum over $a,b,c,d$ in \eqref{e:sigmanablay} is independent of the choice of orthonormal frame $e_1,...,e_n$ at the point $y\in M$, and therefore adopting the abstract index notation convention, we may write
\begin{equation} 
\label{e:sigmanabla}
\sigma_1^G(\nabla_X)=X\otimes 1+\frac{1}{8}\sum_{a,b,c,d} R_{cdab}X_a\br{x}_b\otimes e_ce_d.
\end{equation}
To make the resulting expression more transparent, we introduce the following notation.
\begin{definition}
\label{d:F}
For any $A \in \Gamma(\wedge^2T^*M\otimes \mf{o}(TM))$ define $A^\top \in \Gamma(\wedge^2T^*M\otimes \mf{o}(TM))$ by
\[ g(A^{\top}(X,Y)W,Z)=g(A(W,Z)X,Y).\]
When $B=0$ one has $R^{\top}=R$ by a well-known property of the curvature of the Levi-Civita connection. We also define $R^{\top}_{ab}=g(R^{\top}e_b,e_a) \in \wedge^2T^*M$, the matrix elements of $R^{\top}$ in the orthonormal frame.
\end{definition}
In terms of $R^{\top}$, the Getzler symbol of $\nabla_X$ reads
\begin{equation}
\label{e:GetzlerNabla}
\sigma_1^G(\nabla_X)=X\otimes 1-\frac{1}{4}\sum_{a,b}X_a\br{x}_b\otimes R^{\top}_{ab}.
\end{equation}

\subsubsection{Getzler symbol of $\Delta$, $\Delta-c(\d B)$}
\begin{definition}
Let $D=D^{LC}+c(B)$ where $B \in \Gamma(\wedge^3T^*M)$ is a $3$-form, and recall that we defined $\Delta$ to be the square $D^2$. Let 
\[ \bar{\Delta}=\Delta-c(\d B)= \bar{\Delta}=\nabla^*\nabla+\frac{\kappa}{4}-2|B|^2,\]
where the second expression is Proposition \ref{p:BismutLichnerowicz}.
\end{definition}
If $\d B \ne 0$ then by Example \ref{ex:forms} $o^G(\Delta)=4$ and
\[ \sigma^G_4(\Delta)=1\otimes \d B \in \mf{D}(TM)\otimes \wedge T^*M. \]
More interesting is the operator $\bar{\Delta}$. By Proposition \ref{p:BismutLichnerowicz} and equation \eqref{e:GetzlerNabla}, $o^G(\bar{\Delta})=2$ and
\begin{equation}
\label{e:GetzlerBarDelta}
\sigma^G_2(\bar{\Delta})=-\sum_a \Big(\partial_a\otimes 1-\frac{1}{4}\sum_b \br{x}_b\otimes R^{\top}_{ab}\Big)^2.
\end{equation}

\subsection{Heat kernels}\label{s:heatkernels}
It is well-known (cf. \cite[Theorem 2.30]{BerlineGetzlerVergne}) that the heat operators $e^{-t\Delta}$, $e^{-t\bar{\Delta}}$ have integral kernels $\Theta_t,\bar{\Theta}_t \in C^\infty(M\times M\times (0,\infty),S\boxtimes S^*)$ depending smoothly on $(x,y,t)\in M \times M \times (0,\infty)$, with asymptotic expansions on $U$ as $t \rightarrow 0^+$ (in the space $C^\ell(U,S\boxtimes S^*|_U)$ for any $\ell$),
\begin{align}
\label{e:aexp}
\Theta_t(\br{x})&\sim h_t(\br{x})\sum_{j\ge 0} t^j\Theta_j(\br{x}),\\
\label{e:aexp2}
\bar{\Theta}_t(\br{x})&\sim h_t(\br{x})\sum_{j\ge 0} t^j\bar{\Theta}_j(\br{x}),
\end{align} 
where 
\begin{equation}
\label{e:h}
h_t(\br{x})=(4\pi t)^{-n/2}e^{-\br{x}^2/4t},
\end{equation}
is the Euclidean approximation to the heat kernel. Moreover the expansions remain valid after differentiating both sides with respect to $t$ any number of times. It is convenient to let $\Theta_j=0$ when $j<0$.

The Getzler orders and symbols of the heat kernel coefficients $\bar{\Theta}_j \in C^\infty(U,S\boxtimes S^*|_U)$, $j=0,...,\frac{n}{2}$ can be computed using \eqref{e:GetzlerBarDelta} and Mehler's formula for the solution of the harmonic oscillator. The result is as follows.
\begin{theorem}[\cite{BerlineGetzlerVergne}, Theorem 4.21]
\label{t:Mehler}
For $j=0,...,\frac{n}{2}$ the Getzler order $o^G(\bar{\Theta}_j)=2j$, and the Getzler symbols are given by the generating function
\begin{equation} 
\label{e:Mehler}
\sum_{j=0}^{n/2}t^j\sigma^G_{2j}(\bar{\Theta}_j)=\bigg(1\otimes\tn{det}^{1/2}\Big(\frac{tR^{\top}/2}{\sinh(tR^{\top}/2)}\Big)\bigg)\cdot \tn{exp}\bigg(- \frac{1}{4t}\sum_{a,b} \br{x}_a\br{x}_b\otimes f(R^{\top})_{ab}\bigg)
\end{equation}
where $f(z)=\frac{z}{2}\coth(\frac{z}{2})-1$ and $R^{\top}\in \Gamma(\wedge^2 T^*M \otimes \mf{o}(TM))$ is as in Definition \ref{d:F}.
\end{theorem}
The constant part of the Getzler symbol is obtained by setting $\br{x}=0$ in \eqref{e:Mehler}, resulting in the differential form:
\begin{equation} 
\label{e:fakeAhat}
\sum_{j=0}^{n/2}t^j\sigma^{G,0}_{2j}(\bar{\Theta}_j)=\tn{det}^{1/2}\Big(\frac{tR^{\top}/2}{\sinh(tR^{\top}/2)}\Big).
\end{equation}
As mentioned in Definition \ref{d:F}, when $B=0$, $R^{\top}=R$ is the Riemann curvature tensor of the metric $g$, hence upon setting $t=1$ the form \eqref{e:fakeAhat} becomes the usual Chern-Weil representative of the $\Ahat$-class. In general we have the following formula for $R^\top$, which slightly generalizes \cite[Theorem 1.6]{bismut1989local}.
\begin{proposition}
\label{p:tildeR}
Let $B$ be a $3$-form. Let $\nabla=\nabla^{LC}+B_\mf{o}$ and $\nabla_-=\nabla^{LC}-B_{\mf{o}}$ with curvature tensors $R$, $R_-$ respectively. Then $R^\top=R_-+(\d B)_{\mf{o}}$. In particular when $\d B=0$, $R^\top=R_-$, and \eqref{e:fakeAhat} is the Chern-Weil representative of the $\Ahat$-class (up to factors of $2\pi \i$) constructed using the connection $\nabla_-$.
\end{proposition}
\begin{proof}
Let $\d_{\nabla^{LC}}$ be the exterior covariant differential defined by the Levi-Civita connection. Then
\[ R=(\d_{\nabla^{LC}}+B_{\mf{o}})^2=R^{LC}+\d_{\nabla^{LC}}B_{\mf{o}}+B_{\mf{o}}^2\]
and thus
\begin{equation} 
\label{e:Rtop}
R^\top=R^{LC}+(\d_{\nabla^{LC}}B_{\mf{o}})^\top+(B_{\mf{o}}^2)^\top.
\end{equation}
Let $e_1,...,e_n$ be a local orthonormal frame. Then
\begin{align*} 
g((B_{\mf{o}}^2)^\top(W,X)Y,Z)&=g([B_{\mf{o}}(Y),B_{\mf{o}}(Z)]W,X)\\
&=4\sum_i B(Y,e_i,X)B(Z,W,e_i)-B(Z,e_i,X)B(Y,W,e_i), 
\end{align*}
and if one expands $g(B_{\mf{o}}^2(W,X)Y,Z)$ in the same way one finds the same expression (after using the antisymmetry of $B$ to permute the entries). Therefore $(B_{\mf{o}}^2)^\top=B_{\mf{o}}^2$. On the other hand
\[ \d B(W,X,Y,Z)=(\nabla^{LC}_WB)(X,Y,Z)-(\nabla^{LC}_XB)(W,Y,Z)+(\nabla^{LC}_YB)(W,X,Z)-(\nabla^{LC}_ZB)(W,X,Y), \]
and since $(\nabla^{LC}_X B)_{\mf{o}}=\nabla^{LC}_X B_{\mf{o}}$, we find (grouping the four terms in the expression for $\d B$ in two groups of two):
\[ (\d B)_{\mf{o}}=\d_{\nabla^{LC}}B_{\mf{o}}+(\d_{\nabla^{LC}}B_{\mf{o}})^\top.\]
Thus equation \eqref{e:Rtop} becomes
\[ R^\top=R^{LC}-\d_{\nabla^{LC}}B_{\mf{o}}+B_{\mf{o}}^2+(\d B)_{\mf{o}}=R_-+(\d B)_{\mf{o}}.\]
\end{proof}

\section{The residue cocycle for $D=D^{LC}+c(B)$}
In this section we study the residue cocycle for the spectral triple $(C^\infty(M),L^2(M,S),D)$, where the operator $D=D^{LC}+c(B)$, $B \in \Gamma(\wedge^3T^*M)$. We describe some constraints on what contributions can occur in general, and show how to calculate the cocycle completely when $\d B=0$ using the Getzler calculus. We will use notation introduced in the previous section. In particular recall $\Delta=D^2$, $\bar{\Delta}=\Delta-c(\d B)$, as well as the heat kernels $\Theta_t$, $\bar{\Theta}_t$, the asymptotic expansion coefficients $\Theta_j$, $\bar{\Theta}_j$, and the Euclidean approximation to the heat kernel $h_t(\br{x})=(4\pi t)^{-n/2}e^{-\br{x}^2/4t}$.

According to Proposition \ref{p:heatkernelversion} (and using $[D,a]=c(\d a)$ for $a \in C^\infty(M)$), for $0\le p \le n$ the component $\varphi_p(a_0,...,a_p)$, $a_0,...,a_p\in C^\infty(M)$, of the residue cocycle can be determined from the asymptotic expansion as $t \rightarrow 0^+$ of
\begin{equation} 
\label{e:TrP}
\Tr_s(P_k(a_0,...,a_p)e^{-t\Delta}),
\end{equation}
where $k=(k_1,...,k_p) \in (\bZ_{\ge 0})^p$, $|k|\le n-p$ and
\begin{equation} 
\label{e:Pk}
P_k=P_k(a_0,...,a_p)=a_0\big(\ad_\Delta^{k_1}c(\d a_1)\big)\cdots \big(\ad_\Delta^{k_p}c(\d a_p)\big).
\end{equation}
As $p$ and $a_0,...,a_p$ will be fixed, to simplify notation we will write $P_k$ instead of $P_k(a_0,...,a_p)$ below. It is also convenient to introduce the operator 
\begin{equation} 
\label{e:Pk}
\bar{P}_k=a_0\big(\ad_{\bar{\Delta}}^{k_1}c(\d a_1)\big)\cdots \big(\ad_{\bar{\Delta}}^{k_p}c(\d a_p)\big),
\end{equation}
obtained by replacing $\Delta$ with $\bar{\Delta}=\Delta-c(\d B)$.

The integral kernel of $P_k e^{-t\Delta}$ is $P_k\Theta_t$, where the notation means that $P_k$, $\Theta_t$ are composed as $s$-fibred differential operators. (See Example \ref{ex:diffopfibred} for the sense in which $P_k$ is an $s$-fibred differential operator; another reasonable, though cumbersome, notation would be $P_k(x)\Theta_t(x,y)$ to emphasize that $P_k$ acts along the first factor in the product $M \times M$.) The supertrace \eqref{e:TrP} is given by integration of the pointwise supertrace of the kernel along the diagonal: 
\begin{equation}
\label{e:TrPIntegral}
\Tr_s(P_k e^{-t\Delta})=\int_M \tr_s\big(P_k\Theta_t|_{\br{x}=0}\big) \d V,
\end{equation}
where $\d V$ is the Riemannian measure. Thus computing the asymptotic expansion of \eqref{e:TrP} amounts to studying the low-lying terms in the asymptotic expansion in $t$ of the integrand \eqref{e:TrPIntegral}, and in particular it is enough to work in an arbitrarily small neighborhood of the diagonal in $M \times M$. Using the asymptotic expansion \eqref{e:aexp} of $\Theta_t$,
\begin{equation} 
\label{e:PIexp}
\tr_s(P_k\Theta_t|_{\br{x}=0})\sim (4\pi t)^{-n/2}\sum_{j\ge 0} t^j \tr_s(h_t^{-1}P_kh_t\Theta_j|_{\br{x}=0}), 
\end{equation}
and likewise for $\tr_s(\bar{P}_k\bar{\Theta}_t|_{\br{x}=0})$.

\begin{lemma}
\label{l:hPh}
Let $\O$ be an $s$-fibred differential operator of order $m$. Then
\begin{equation} 
\label{e:expO}
h_t^{-1}\O h_t=\sum_{\ell=0}^m t^{-m+\ell}\O_\ell 
\end{equation}
where $\O_\ell$ is an $s$-fibred differential operator of order $\ell$. For $\ell<\lceil m/2\rceil$, $\O_\ell$ vanishes on the diagonal (in the sense of Definition \ref{d:$r$-fibreddo}).
\end{lemma}
\begin{proof}
Since $h_t(\br{x})=(4\pi t)^{-n/2}e^{-\br{x}^2/4t}$, the coefficient of $t^{-m+\ell}$ comes from applying $m-\ell$ derivatives to $h_t$, leaving $m-(m-\ell)=\ell$ derivatives. This proves the claim regarding the order of $\O_\ell$.

Let $\Psi \in \Gamma(S\boxtimes S^*|_U)$ and set $\O(t)=h_t^{-1}\O h_t$. Fix $t$ and consider the re-scaling $\phi_{0,u}\colon t \mapsto ut$. Along the diagonal, the re-scaled section $\phi_{0,u}^* \O(t)\Psi|_{\br{x}=0}$ has an asymptotic expansion in $u$ as $u \rightarrow 0^+$, with some lowest power $u^{-m+p_\Psi}$ where $p_\Psi \ge 0$. Taking the infimum of $-m+p_\Psi$ over all choices of $\Psi$ yields the lowest power $-m+p \in \bZ$ of $t$ in \eqref{e:expO} such that the $s$-fibred differential operator $\O_p$ does not vanish on the diagonal.

On the other hand we may obtain a lower bound on $-m+p$ by considering the growth rate of the section $\O(t)\Psi \in \Gamma(S\boxtimes S^*|_U)$ (no restriction to the diagonal $\br{x}=0$) under the combined re-scaling $\phi_u \colon (t,\br{x}) \mapsto (ut,u^{1/2}\br{x})$. Clearly $\phi_u^*e^{-\br{x}^2/4t}=e^{-\br{x}^2/4t}$, and as $\O$ has order $m$, the section $\O(t)\Psi$ grows at most at the rate $u^{-m/2}$ under the re-scaling. Hence $-m+p\ge -m/2$ and so $p \ge m/2$. Since $p$ is an integer, $p \ge \lceil m/2\rceil$.
\end{proof}

In particular the above Lemma applies to the order $|k|$ differential operators $P_k$, $\bar{P}_k$. For the operator $\O=\bar{P}_k$ we will also need the Getzler orders of the operators $\bar{P}_{k,\ell}=\O_\ell$.

\begin{proposition}
\label{p:Deltah}
In normal coordinates
\[ h^{-1}_t\Delta h_t=\Delta+\frac{1}{t}\nabla_\E+\frac{n}{2t}+\frac{1}{t}\E(\log \varrho)-\frac{\br{x}^2}{4t^2} \]
where $\varrho=|g|^{1/4}$ (see Example \ref{ex:varrho}). The same formula holds with $\Delta$ replaced by $\bar{\Delta}$.
\end{proposition}
\begin{proof}
One has
\[ h^{-1}_t\Delta h_t=\Delta+h_t^{-1}[\Delta,h_t]=\Delta-2h^{-1}_t\nabla_{\nabla h_t}+h^{-1}_t\Delta h_t.\]
On the other hand
\[ h^{-1}_t\nabla h_t=-\frac{1}{2t}\E \]
and a short calculation in normal coordinates (cf. \cite[p.100]{roe1998elliptic}) shows that
\[ h_t^{-1}(\Delta h_t)=-\frac{\br{x}^2}{4t}+\frac{n}{2t}+\frac{1}{t}\E(\log \varrho).\]
\end{proof}

\begin{lemma}
\label{l:iteratedcommutator}
Let $m>0$ and $f \in C^\infty(M)$. On $U \subset M \times M$ we have 
\[ \ad^m_{h^{-1}_t\bar{\Delta} h_t}(c(\d f))=t^{-m}\O_0+t^{-m+1}\O_1+\cdots+t^0\O_m. \]
where $\O_\ell$ is an $s$-fibred differential operator with Getzler order $o^G(\O_\ell)\le 2\ell$.
\end{lemma}
\begin{proof}
Proceed by induction on $m$. For the base case $m=1$ we have 
\[[h^{-1}_t\bar{\Delta} h_t,c(\d f)]=t^{-1}[\nabla_\E,c(\d f)]+[\bar{\Delta},c(\d f)]=t^{-1}\O_0+\O_1.
\]
The operator $\O_0=c(\nabla_\E \d f)$ has Getzler order $o^G(\O_0)=1-1=0$, since $o^G(c(\alpha))=1$ for $\alpha \in \Omega^1(M)$ but $\E$ vanishes to order $1$ on the diagonal (contributing $-1$, see Example \ref{ex:Getzlerfunctions}).  Since $o^G(\bar{\Delta})=2$, $o^G(c(\d f))=1$ we have $o^G(\O_1)\le 3$. But in fact $o^G(\O_1)=2$ because the Getzler symbols 
\[ \sigma^G_2(\bar{\Delta})=-\sum_a \Big(\partial_a\otimes 1-\frac{1}{4}\sum_b \br{x}_b\otimes R^{\top}_{ab}\Big)^2, \qquad \sigma^G_1(c(\d f))=1\otimes \d f \]
commute. This establishes the base case. For the inductive step, suppose
\[ \ad^{m-1}_{h^{-1}_t\bar{\Delta} h_t}(c(\d a))=t^{-(m-1)}\O_0+t^{-(m-1)+1}\O_{1}+\cdots+t^0\O_{m-1}=\sum_\ell t^{-(m-1)+\ell}\O_{\ell},\]
where $o^G(\O_j)=2j$. Then
\[ \ad^m_{h^{-1}_t\bar{\Delta} h_t}(c(\d f))=t^{-(m-1)}[h^{-1}_t\bar{\Delta} h_t,\O_0]+t^{-(m-1)+1}[h^{-1}_t\bar{\Delta} h_t,\O_{1}]+\cdots+t^0[h^{-1}_t\bar{\Delta} h_t,\O_{m-1}].\]
By Proposition \ref{p:Deltah}, $h^{-1}_t\bar{\Delta} h_t=\bar{\Delta}+t^{-1}T+t^{-2}F$ where $o^G(T)=0$, $o^G(F)=-2$. The Getzler orders are $o^G([\bar{\Delta},\O_\ell])=2\ell+2$, $o^G([T,\O_\ell])=2\ell$, and $o^G([F,\O_\ell])=2\ell-2$. Hence a typical term
\[ t^{-(m-1)+\ell}[h^{-1}_t\bar{\Delta} h_t,\O_{\ell}]=t^{-m+(\ell+1)}\O_{\ell+1}'+t^{-m+\ell}\O_{\ell}'+t^{-m+(\ell-1)}\O_{\ell-1}', \]
has Getzler orders as claimed, completing the inductive step.
\end{proof}

The following summarizes the result of applying the previous two lemmas to $P_k$, $\bar{P}_k$.
\begin{corollary}
\label{c:hPh}
On $U \subset M \times M$ we have
\[ h^{-1}_tP_kh_t=\sum_{\ell = 0}^{|k|}t^{-|k|+\ell}P_{k,\ell}, \]
where each $P_{k,\ell}$ is an $s$-fibred differential operator of order $o(P_{k,\ell})=\ell$. For $\ell < \lceil |k|/2 \rceil$, $P_{k,\ell}$ vanishes on the diagonal. There is a similar expansion for $h^{-1}_t\bar{P}_kh_t$, with the additional property the Getzler order $o^G(\bar{P}_{k,\ell})\le 2\ell+N_{=0}(k)$, where $0\le N_{=0}(k)\le p$ is the number of indices $i$ such that $k_i=0$.
\end{corollary}
\begin{proof}
The claims regarding the order of $P_{k,\ell}$ and its vanishing along the diagonal are immediate consequences of Lemma \ref{l:hPh}. By Lemma \ref{l:iteratedcommutator}, the Getzler order $o^G(\bar{P}_{k,\ell})$ is $2\ell$ if $k_i\ne 0$ for all $i=1,...,p$. For each index $i$ such that $k_i=0$, the Getzler order count becomes $1$ larger than this because $o^G(c(\d a_i))=1$ instead of $0$.
\end{proof}

\begin{corollary}
\label{c:relevantterms}
There is an asymptotic expansion
\[\Tr_s(P_ke^{-t\Delta}) \sim (4\pi)^{-n/2}t^{-n/2-\lfloor |k|/2 \rfloor}\sum_{j\geq 0}\sum_{r = 0}^{\lfloor |k|/2 \rfloor}t^{j+r}\int_M\tr_s(P_{k,r+\lceil |k|/2 \rceil}\Theta_j|_{\br{x}=0})\d V.\]
The component $\varphi_p$ of the residue cocycle is given by
\[ \varphi_p(a_0,...,a_p)=\sum_{|k|\le n-p} \frac{c_{pk}'}{(4\pi)^{n/2}}\sum_{j=(n-p)/2-|k|}^{(n-p)/2-\lceil |k|/2 \rceil}
\int_M\tr_s(P_{k,(n-p)/2-j}\Theta_j|_{\br{x}=0})\d V.\]
\end{corollary}
\begin{proof}
The expansion for $\Tr_s(P_ke^{-t\Delta})$ follows from substituting the expansion from Corollary \ref{c:hPh} into \eqref{e:PIexp} and making the change of variables $r=\ell-\lceil |k|/2 \rceil$. The formula for $\varphi_p$ is an immediate consequence of Proposition \ref{p:heatkernelversion}.
\end{proof}

\begin{theorem}
\label{t:dbzero}
If $\d B=0$, then
\[ \varphi_p(a_0,...,a_p)=\frac{(2\pi \i)^{-n/2}}{p!}\int_M a_0 \d a_1\cdots \d a_p \cdot \tn{det}^{1/2}\Big(\frac{R_-/2}{\sinh(R_-/2)}\Big)_{[n-p]},\]
where $R_-$ is the curvature of the connection $\nabla_-=\nabla^{LC}-B_{\mf{o}}$.
\end{theorem}
\begin{proof}
Let $\Psi_{j,k}=P_{k,(n-p)/2-j}\Theta_j|_U\in \Gamma(S\boxtimes S^*|_U)$. The $n=\dim(M)$ order Getzler symbol of $\Psi_{j,k}$ is the $n$-form part of the Clifford symbol of $\Psi_{j,k}|_{\br{x}=0}$. By \cite[Proposition 3.21]{BerlineGetzlerVergne},
\[ \tr_s(\Psi_{j,k}|_{\br{x}=0})=(-2\i)^{n/2}\scr{B}(\sigma^{G,0}_n(\Psi_{j,k})) \]
where $\scr{B}\colon \wedge^n\!T^*M \rightarrow M\times \bR$ is the Berezin integral (\cite[p.40]{BerlineGetzlerVergne}) determined by the orientation and the metric. Therefore by Corollary \ref{c:relevantterms}
\[ \varphi_p(a_0,...,a_p)=\sum_{|k|\le n-p} c_{pk}'(2\pi \i)^{-n/2}\sum_{j=(n-p)/2-|k|}^{(n-p)/2-\lceil |k|/2 \rceil}\int_M\scr{B}(\sigma^{G,0}_n(\Psi_{j,k}))\d V.\]
When $\d B=0$, $P_k=\bar{P}_k$, $\Theta_j=\bar{\Theta}_j$ hence by Theorem \ref{t:Mehler}, $o^G(\Theta_j)=2j$. By Corollary \ref{c:hPh}, $o^G(P_{k,\ell})=2\ell+N_{=0}(k)$ where $0\le N_{=0}(k)\le p$ is the number of indices $i\in \{1,...,p\}$ such that $k_i=0$. Therefore
\[ o^G(\Psi_{j,k})\le n-p+N_{=0}(k).\]
When at least one $k_i \ne 0$, $N_{=0}(k)<p$ and $o^G(\Psi_{j,k})<n$, and so $\sigma^G_n(\Psi_{j,k})=0$. Otherwise if $k=(0,...,0)$ then $c_{pk}'=1/p!$ and by Example \ref{ex:forms} and equation \eqref{e:fakeAhat},
\[ \sigma^{G,0}_n(\Psi_{j,k})=a_0\d a_1\cdots \d a_p \cdot \tn{det}^{1/2}\left(\frac{R^{\top}/2}{\sinh(R^{\top}/2)}\right)_{[n-p]}. \]
The result follows from this and Proposition \ref{p:tildeR}.
\end{proof}

When Theorem \ref{t:CMindex} is specialized to the triple $(C^\infty(M),L^2(M,S),D)$ and the idempotent $e=1$, the result is the Atiyah-Singer formula:
\[ \index(D)=\varphi_0(1)=(2\pi \i)^{-n/2}\int_M \tn{det}^{1/2}\Big(\frac{R_-/2}{\sinh(R_-/2)}\Big).\]

\section{Residue cocycle calculations when $n=4$}\label{s:n4}
In this section we compute the Connes-Moscovici cocycle completely for the operator $D=D^{LC}+c(B)$, $B \in \Gamma(\wedge^3T^*M)$ when $\d B \ne 0$ and the dimension $n=4$. The case $p=4$ may be disposed of immediately: by Corollary \ref{c:relevantterms}, only the $|k|=0$, $j=(4-4)/2=0$ term contributes, thus
\[\varphi_4(a_0,a_1,a_2,a_3,a_4)=\frac{(2\pi i)^{-2}}{4!}\int_M a_0\d a_1 \d a_2\d a_3\d a_4.\] 
The remaining two components $\varphi_0$, $\varphi_2$ of the residue cocycle are computed in the sections below.

\subsection{Recursion relation for the heat kernel coefficients}\label{s:recrel}
There are well-known recursion relations for the heat kernel asymptotic expansion coefficients $\Theta_j$, $\bar{\Theta}_j$ that we briefly recall here. The setup of Section \ref{s:Getzler} will be used in the calculations. In particular we use the inverse of the Riemannian exponential map to identify $U \subset M \times M$ with a neighborhood of the $0$-section $\br{U}\subset TM$, and $\nabla$-parallel translation along radial geodesics is used to identify $S\boxtimes S^*|_U$ with $s^*\End(S)$. Thus for example $\Theta_j$, $\bar{\Theta}_j$, $c(\d B)$ are identified with $\End(\pr^*S)$-valued smooth functions on $\br{U}$. Under the identifications the operator $\nabla_\E$ becomes $\E$. The operators $\Delta$, $\bar{\Delta}$, $c(\d B)$ are viewed as $s$-fibred differential operators (see Example \ref{ex:diffopfibred}). With this understanding, the recurrence relation satisfied by the coefficients $\Theta_j$ is (cf. \cite[Theorem 2.26]{BerlineGetzlerVergne}):
\[ \E \Theta_j+(j+\E (\log \varrho))\Theta_j=-\Delta \Theta_{j-1}, \quad j\ge 1; \qquad \Theta_0=\varrho^{-1} \]
where recall $\varrho=|g|^{1/4}$, with $|g|$ the determinant of the Riemannian metric in normal coordinates (see Example \ref{ex:varrho}). The solutions are given recursively by the formula
\[ \Theta_j(\br{x})=-\varrho^{-1}(\br{x})\int_0^1 t^{j-1}\varrho(t\br{x})\Delta \Theta_{j-1}(t\br{x})\d t.\]
Of course the same equations hold with $\Theta_j$, $\Delta$ replaced by $\bar{\Theta}_j$, $\bar{\Delta}$. Using $\Delta=\bar{\Delta}+c(\d B)$, for the first few terms $j=0,1,2$ we find
\begin{equation} 
\label{e:theta01}
\Theta_0=\bar{\Theta}_0=\varrho^{-1}, \qquad \Theta_1=\bar{\Theta}_1+\Theta_1^B
\end{equation}
where
\begin{equation} 
\label{e:theta1B}
\Theta_1^B(\br{x})=-\varrho^{-1}(\br{x})\int_0^1 c(\d B)_{t\br{x}}\d t, 
\end{equation}
and
\begin{align}
    \Theta_2(\br{x})&=-\varrho^{-1}(\br{x})\int_0^1t\varrho(t\br{x})(\Delta\Theta_1)(t\br{x})\d t \nonumber\\
    &=\bar{\Theta}_2(\br{x})-\varrho^{-1}(\br{x})\int_0^1 t\varrho(t\br{x})\Big(c(\d B)_{t\br{x}}\bar{\Theta}_1(t\br{x})+(\Delta\Theta^B_1)(t\br{x})\Big)\d t.\label{j=2}
\end{align}

In a normal coordinate neighborhood $U_y=U\cap (M \times \{y\})$, the Laplacian $\Delta$ is given by
\begin{equation} 
\label{e:DeltaLocal}
\Delta=-\sum_{i=1}^n\Big(\big(e_i+\omega(e_i)\big)^2-\big(\nabla_{e_i}e_i+\omega(\nabla_{e_i}e_i)\big)\Big)+\frac{\kappa}{4}+c(\d B)-2|B|^2,
\end{equation}
where $\omega \in \Omega^1(U_y,\End(S_y))\simeq \Omega^1(U_y,\bC l(T_y^*M))$ is the connection $1$-form relative to the identification $S|_{U_y}\simeq U_y\times S_y$ given by $\nabla$-parallel translation along radial geodesics, and $e_1,...,e_n$ is a local orthonormal frame on $U_y$ obtained from an orthonormal basis of $T_yM$ by $\nabla$-parallel translation along radial geodesics. In particular, at the origin $\br{x}_y=0$ of the chart, $\omega(e_i)|_{\br{x}_y=0}=0$, $\nabla_{e_i}e_i|_{\br{x}_y=0}=0$, $e_i \omega(e_i)|_{\br{x}_y=0}=0$ (this last follows from \eqref{e:omegalinear} and skew-symmetry of $R$), and hence near $y \in M$,
\begin{equation}
\label{e:DeltaLinear}
\Delta=-\sum_{i=1}^n e_i^2+\frac{\kappa_y}{4}+c(\d B)_y-2|B_y|^2+O(|\br{x}_y|\partial),
\end{equation}
where $O(|\br{x}_y| \partial)$ denotes a $1$-st order differential operator with coefficients that vanish at $y$. The operator $\bar{\Delta}$ has a similar expression, leaving out the $c(\d B)_y$ term.

\subsection{Computation of $\varphi_0$ ($p=0$)}
When $p=0$, $P_k=a_0$ and the commutators (indexed by $k=(k_1,...,k_p)$) are absent. By Corollary \ref{c:relevantterms},
\begin{equation*}
    \label{e:p=0}
    \varphi_0(a_0)=(4\pi)^{-2}\int_M a_0\tr_s\big(\Theta_2|_{\br{x}=0}\big)\d V.
\end{equation*}
The pointwise supertrace $\tr_s(\Theta_2|_{\br{x}=0})$ can be computed explicitly in terms of the differential forms $R^{\top}$ (see Definition \ref{d:F}) and $B$, leading to the following.
\begin{theorem}
\label{t:phi0}
Let $\dim(M)=4$ and $D=D^{LC}+c(B)$ where $B \in \Gamma(\wedge^3 TM)$. The component $\varphi_0$ of the residue cocycle is
\[ \varphi_0(a_0)=(2\pi \i)^{-2}\int_M a_0\bigg(\tn{det}^{1/2}\Big(\frac{R^{\top}/2}{\sinh(R^{\top}/2)}\Big)+\big(\frac{\kappa}{12}-2|B|^2\big)\d B+\frac{1}{6}\d \d^* \d B\bigg), \]
where $\kappa$ is the scalar curvature and $R^{\top}$ is the differential form of Definition \ref{d:F}.
\end{theorem}
\begin{proof}
The term involving $R^{\top}$ in the statement of Theorem \ref{t:phi0} comes from $\tr_s(\bar{\Theta}_2|_{\br{x}=0})$; see the proof of Theorem \ref{t:dbzero}. On the other hand $(\Theta_2-\bar{\Theta}_2)(\br{x})$ is given by \eqref{j=2}:
\begin{equation} 
\label{e:thetatheta}
-\varrho^{-1}(\br{x})\int_0^1 t\varrho(t\br{x})\Big(c(\d B)_{t\br{x}}\bar{\Theta}_1(t\br{x})+(\Delta\Theta^B_1)(t\br{x})\Big)\d t.
\end{equation}
Evaluating at $\br{x}=0$, using $\varrho(0)=1$, and performing the integral over $t$, \eqref{e:thetatheta} becomes
\begin{equation}
\label{e:thetatheta2}
(\Theta_2-\bar{\Theta}_2)|_{\br{x}=0}=-\frac{1}{2}\Big(c(\d B)\bar{\Theta}_1|_{\br{x}=0}+(\Delta\Theta^B_1)|_{\br{x}=0}\Big).
\end{equation}
We calculate the supertrace of the two terms I, II of \eqref{e:thetatheta2} in turn.
\medskip

\noindent \textbf{Term I}: we claim that 
\[ -\frac{(4\pi)^{-2}}{2}\tr_s(c(\d B)\bar{\Theta}_1|_{\br{x}=0})=\frac{(2\pi \i)^{-2}}{2}\big(\frac{\kappa}{12}-2|B|^2\big)\scr{B}(\d B), \]
where $\scr{B}\colon \wedge^nT^*M\rightarrow M \times \bR$ is the Berezin integral. To show this we compute at $y \in M$ using equations \eqref{e:theta01}, \eqref{e:DeltaLinear},
\begin{equation} 
\label{e:I1}
\bar{\Theta}_1|_{\br{x}_y=0}=-\bar{\Delta} \varrho^{-1}_y|_{\br{x}_y=0}=\sum_{i=1}^4(e_i)^2\varrho_y^{-1}\big|_{\br{x}_y=0}-\frac{\kappa_y}{4}+2|B_y|^2.
\end{equation}
Using the Taylor series of $g_{ij}$ in normal coordinates (cf. \cite[Proposition 1.28]{BerlineGetzlerVergne}), one has
\begin{equation}
\label{e:Laplacianvarrho}
\sum_{i=1}^4(e_i)^2\varrho^{-1}_y\big|_{\br{x}_y=0}=\frac{\kappa_y}{6},
\end{equation}
at any $y \in M$. Thus taking the supertrace (using \cite[Proposition 3.21]{BerlineGetzlerVergne}),
\[ \tr_s(c(\d B)\bar{\Theta}_1|_{\br{x}=0})=(-2\i)^2\Big(\frac{\kappa}{6}-\frac{\kappa}{4}+2|B|^2\Big)\scr{B}(\d B)=-(-2\i)^2\Big(\frac{\kappa}{12}-2|B|^2\Big)\scr{B}(\d B), \]
which gives the claim.
\medskip

\noindent \textbf{Term II}: we claim that
\[ -\frac{(4\pi)^{-2}}{2}\tr_s(\Delta\Theta^B_1)|_{\br{x}=0}=\frac{(2\pi \i)^{-2}}{2}\big(\frac{\kappa}{12}-2|B|^2\big)\scr{B}(\d B)+\frac{(2\pi \i)^{-2}}{6}\scr{B}(\d \d^* \d B).\]
To verify this we compute at $y \in M$ using \eqref{e:theta1B}, \eqref{e:DeltaLinear},
\begin{align} 
\label{e:II}
\Delta\Theta_1^B|_{\br{x}_y=0}&=\Big(-\sum_i e_i^2+\frac{\kappa_y}{4}+c(\d B)_y-2|B_y|^2\Big)\Big(\varrho_y^{-1}\int_0^1 c(\d B)_{t\br{x}_y}\d t\Big)\Big|_{\br{x}_y=0}\nonumber \\ 
&=-\sum_i e_i^2 c(\d B)\big|_{\br{x}_y=0}\int_0^1 t^2 \d t+\big(\frac{\kappa_y}{12}+c(\d B)_y-2|B_y|^2\big)c(\d B)_y \nonumber \\
&=\frac{1}{3}c(\d \d^*\d B)_y+\big(\frac{\kappa_y}{12}+c(\d B)_y-2|B_y|^2\big)c(\d B)_y
\end{align}
where to obtain the second equality we used $\varrho_y|_{\br{x}_y=0}=1$, $e_i\varrho_y|_{\br{x}_y=0}=0$, \eqref{e:Laplacianvarrho}, and for the third equality we used that $-\sum_i e_i^2 \nu|_{\br{x}_y=0}=(\d \d^* \nu)_y$ on $4$-forms $\nu$ (since $\dim(M)=4$). Another consequence of $\dim(M)=4$ is that $c(\d B)^2$ is scalar, hence $-\tr_s(c(\d B)^2)=0$. Taking the supertrace of \eqref{e:II} as in the computation of Term I gives the claim.
\end{proof}

By expressing $R^\top$ in terms of the curvature $R_-$ of $\nabla_-=\nabla^{LC}-B_{\mf{o}}$, further simplification of the expression in Theorem \ref{t:phi0} is possible.
\begin{theorem}
\label{t:phi0Rminus}
Let $\dim(M)=4$ and $D=D^{LC}+c(B)$ where $B \in \Gamma(\wedge^3 TM)$. The component $\varphi_0$ of the residue cocycle is
\[ \varphi_0(a_0)=(2\pi \i)^{-2}\int_M a_0\bigg(\tn{det}^{1/2}\Big(\frac{R_-/2}{\sinh(R_-/2)}\Big)+\frac{1}{6}\d \d^* \d B\bigg). \]
\end{theorem}
\begin{proof}
Using the identities
\[ \tn{det}^{1/2}\Big(\frac{R^{\top}/2}{\sinh(R^{\top}/2)}\Big)_{[4]}=-\frac{1}{48}\tr(R^{\top,2}) \]
and (see the proof of Proposition \ref{p:tildeR})
\[ R^\top=R_-+(\d B)_{\mf{o}}, \qquad R_-=R^{LC}-\d_{\nabla^{LC}}B_{\mf{o}}+B_{\mf{o}}^2,\]
one has
\[ \tn{det}^{1/2}\Big(\frac{R^{\top}/2}{\sinh(R^{\top}/2)}\Big)_{[4]}=-\frac{1}{48}\tr\Big(R_-^2+2(R^{LC}-\d_{\nabla^{LC}}B_{\mf{o}}+B_{\mf{o}}^2)(\d B)_\mf{o}+(\d B)_\mf{o}^2\Big). \]
By a short calculation one finds that $\tr((\d B)_\mf{o}^2)=0$, $\tr((\d_{\nabla^{LC}}B_{\mf{o}})(\d B)_\mf{o})=0$ while
\[ \tr(R^{LC}(\d B)_\mf{o})=2\kappa \, \d B, \qquad \tr(B_\mf{o}^2(\d B)_\mf{o})=-48|B|^2 \d B. \]
\end{proof}

\subsection{Computation of $\varphi_2$ ($p=2$)}
By Corollary \ref{c:relevantterms},
\begin{equation} 
\label{e:varphi2}
\varphi_2(a_0,a_1,a_2)=\sum_{|k|\le 2}\frac{c_{2k}'}{(4\pi)^2}\sum_{j=1-|k|}^{1-\lceil |k|/2\rceil} \int_M \tr_s(P_{k,1-j}\Theta_j|_{\br{x}=0})\d V,
\end{equation}
and sum ranges over $k=(k_1,k_2)$ with $k_1,k_2\ge 0$.

\begin{lemma}
\label{l:k0}
The $k=(0,0)$ contribution in \eqref{e:varphi2} is
\[ \frac{(2\pi \i)^{-2}}{2}\int_Ma_0 g(\d a_1,\d a_2)\d B. \]
\end{lemma}
\begin{proof}
If $|k|=0$ then we only have the $j=1$ term in \eqref{e:varphi2}, which reads
\[ \frac{(4\pi)^{-2}}{2} \int_M a_0\tr_s(c(\d a_1)c(\d a_2)\Theta_1|_{\br{x}=0})\d V.\]
We may write $\Theta_1=\bar{\Theta}_1+(\Theta_1-\bar{\Theta}_1)$. The contribution of $\bar{\Theta}_1$ was computed in Theorem \ref{t:dbzero}, and is trivial in this case because the dimension $n=4$: $\d a_1 \d a_2$ is a $2$-form and $\tn{det}^{1/2}((R^\top/2)/\sinh(R^\top/2))$ has vanishing $2$-form component. On the other hand by equation \ref{e:theta1B} and since $\varrho|_{\br{x}=0}=1$,
\[ (\Theta_1-\bar{\Theta}_1)|_{\br{x}=0}=c(\d B).\]
Taking the supertrace yields
\[ \tr_s(c(\d a_1)c(\d a_2)c(\d B))=(-2\i)^2g(\d a_1,\d a_2)\scr{B}(\d B),\]
which gives the lemma.
\end{proof}

\begin{lemma}
The $|k|=2$ contribution in \eqref{e:varphi2} vanishes.
\end{lemma}
\begin{proof}
The contribution involves the supertrace $\tr_s(P_{k,1}\Theta_0|_{\br{x}=0})$. The $s$-fibred differential operators $P_k$, $\bar{P}_k$ have order $|k|$ and the same symbol (in the classical sense), hence $P_k-\bar{P}_k$ is an $s$-fibred differential operator of order $|k|-1$. By Lemma \ref{l:hPh} applied to $P=P_k-\bar{P}_k$, it follows that the $s$-fibred differential operator $P_{k,1}-\bar{P}_{k,1}$ vanishes on the diagonal, and hence $\tr_s(P_{k,1}\Theta_0|_{\br{x}=0})=\tr_s(\bar{P}_{k,1}\Theta_0|_{\br{x}=0})$. But $\Theta_0=\bar{\Theta}_0$, and so this supertrace vanishes by the Getzler order calculations in Theorem \ref{t:dbzero}.
\end{proof}

\begin{theorem}
\label{t:p2}
Let $\dim(M)=4$ and $D=D^{LC}+c(B)$ where $B \in \Gamma(\wedge^3 TM)$. The component $\varphi_2$ of the residue cocycle is
\[ \varphi_2(a_0,a_1,a_2)=\frac{(2\pi \i)^{-2}}{6}\int_M a_0 g(\d a_1,\d a_2)\d B.\]
\end{theorem}
\begin{proof}
The $k=(0,0)$ term of \eqref{e:varphi2} was calculated in Lemma \ref{l:k0}. By the lemmas above, the only remaining contributions in \eqref{e:varphi2} come from $k=(1,0)$, $(0,1)$ and $j=0$. We have
\[ P_{(1,0),1}=\bar{P}_{(1,0),1}+a_0[c(\d B),c(\d a_1)]c(\d a_2), \quad P_{(0,1),1}=\bar{P}_{(0,1),1}+a_0c(\d a_1)[c(\d B),c(\d a_2)].\]
Inserting these expressions in \eqref{e:varphi2}, the supertraces involving $\bar{P}_{k,1}$ vanish by the Getzler order considerations in the proof of Theorem \ref{t:dbzero}. On the other hand $[c(\d B),c(\d a)]=-2c(\d a)c(\d B)$ since $\dim(M)=4$ and $\d B$ is a $4$-form. Consequently
\[ (4\pi)^{-2}\tr_s(P_{(1,0),1}\Theta_0|_{\br{x}=0})=-2(2\pi\i)^{-2}a_0g(\d a_1,\d a_2)\scr{B}(\d B) \]
\[ (4\pi)^{-2}\tr_s(P_{(1,0),1}\Theta_0|_{\br{x}=0})=2(2\pi \i)^{-2}a_0g(\d a_1,\d a_2)\scr{B}(\d B). \]
Combining these terms with the combinatorial prefactors $c_{1,(1,0)}'=-\frac{1}{6}$, $c_{1,(0,1)}'=-\frac{1}{3}$ and adding to the $k=(0,0)$ term from Lemma \ref{l:k0} yields the result.
\end{proof}

\subsection{An index formula}
When Theorem \ref{t:CMindex} is specialized to the triple $(C^\infty(M),L^2(M,S),D)$ with $n=\dim(M)=4$, and the idempotent $e=1$, we recover (since $\d \d^* \d B$ is exact):
\[ \index(D)=(2\pi \i)^{-2}\int_M \tn{det}^{1/2}\Big(\frac{R_-/2}{\sinh(R_-/2)}\Big). \]

\section{Residue cocycle calculations when $n=6$}
In this brief section we describe the outcome of our calculations of the $\varphi_0$ component of the residue cocycle for the operator $D=D^{LC}+c(B)$ (with $\d B\ne 0$) in dimension $n=6$. Already in this dimension the calculations are considerably more involved than in the previous section.

The $\varphi_0$ term is given in terms of the $3$-rd term $\Theta_3$ in the asymptotic expansion of the heat kernel $e^{-t\Delta}$:
\begin{equation*}
    \label{e:p=0}
    \varphi_0(a_0)=(4\pi)^{-3}\int_M a_0\tr_s\big(\Theta_3|_{\br{x}=0}\big)\d V.
\end{equation*}
As in Section \ref{s:n4}, $\Theta_3=\bar{\Theta}_3+\Theta_3^B$ where $\bar{\Theta}_3$ is the asymptotic expansion coefficient for $\bar{\Delta}=\Delta-c(\d B)$. In fact $\tr_s(\bar{\Theta}_3|_{\br{x}=0})=0$ in this case as $(x/2)/\sinh(x/2)$ is an even function hence the top part of the Chern-Weil form representing the $\Ahat$-class vanishes for degree reasons. We computed $\Theta_3^B$ directly using the recursion relations in Section \ref{s:recrel}. A useful basic Clifford algebra fact that allows to eliminate several terms is that if $\alpha_1,\alpha_2 \in \Omega^4(M)$ (and $\dim(M)=6$) then $c(\alpha_1)c(\alpha_2)=c(\alpha)$ where $\alpha$ is an inhomogeneous sum of forms with degrees $0,2,4$; moreover if $\alpha_1=\alpha_2$ then only degrees $0,4$ appear. Thus for example the operators $c(\d B)^2$, $c(\d B)^3$ (amongst others involving the curvature form, etc.) do not appear in the resulting formula as they have vanishing local supertrace. Another ingredient in the calculation is the method explained in \cite[Proposition 1.18, Proposition 1.28]{BerlineGetzlerVergne}, \cite[Appendix II]{atiyah1973heat} for computing the Taylor expansion coefficients of the metric and connection in a synchronous frame over a geodesic coordinate patch. The result of the calculations is expressed here using abstract index notation (on $\mf{o}(TM)$ indices) and the summation convention.
\begin{theorem}
\label{t:p6}
In dimension $n=6$ the $\varphi_0$ term in the residue cocycle for the operator $D=D^{LC}+c(B)$ is
\[ \varphi_0(a_0)=\frac{(2\pi \i)^{-3}}{18}\int_M a_0g^{ab}g^{cd}(\frac{1}{2}\nabla_a R^\top_{bc}+g^{ef}B_{ace}R^\top_{bf})\wedge \nabla_d \d B. \]
\end{theorem}

\bibliographystyle{amsplain}

\providecommand{\bysame}{\leavevmode\hbox to3em{\hrulefill}\thinspace}
\providecommand{\MR}{\relax\ifhmode\unskip\space\fi MR }
\providecommand{\MRhref}[2]{%
  \href{http://www.ams.org/mathscinet-getitem?mr=#1}{#2}
}
\providecommand{\href}[2]{#2}

\end{document}